\documentclass[a4paper,12pt]{article}

\usepackage{amsmath}
\usepackage{amssymb}
\usepackage{amsthm} 
\usepackage[T1]{fontenc}
\usepackage[bf,small,tableposition=top]{caption}
\usepackage{subfig}
\usepackage{setspace}
 \usepackage{multirow}
\usepackage{graphicx}
\usepackage{epstopdf}
\usepackage{array}
\usepackage{lscape}
\usepackage{soul}
\usepackage{bm}
\usepackage{textcomp}
\usepackage{enumitem}
\usepackage{algorithmicx}
\usepackage{lettrine}
\usepackage{tikz}
\usepackage{makeidx}
\usepackage[algoruled,linesnumbered,noline]{algorithm2e}
\usepackage{booktabs,caption,fixltx2e}
\usepackage[flushleft]{threeparttable}
\usepackage{rotating}
\usepackage[sectionbib, super, numbers]{natbib}
\usepackage{tikz}
\usetikzlibrary{positioning,decorations.pathreplacing,arrows,shapes,trees}

\theoremstyle{plain}
\newtheorem{theorem}{Theorem}
\newtheorem{lemma}[theorem]{Lemma}
\newtheorem{rmk}{Remark}
\theoremstyle{remark}

\begin{document}

\begin{center}
 \Large\bf{ Statistical inference and Bayesian optimal life-testing plans under Type-II unified hybrid censoring scheme}
\end{center}
 \begin{center}
 \bf {\footnotesize Tanmay Sen$^{1}$, Ritwik Bhattacharya\footnote[2]{\textit{Corresponding author}: Ritwik Bhattacharya (ritwik.bhatta$@$gmail.com)}, Biswabrata Pradhan$^{3}$ and Yogesh Mani Tripathi$^4$}
 \end{center}
  \begin{center}
 \noindent\textit{\scriptsize $^1$ Data Science and Artificial Intelligence, L \& T Technology Services, Karnataka 560045, India }\\
  \noindent\textit{\scriptsize $^2$ Department of Industrial Engineering, School of Engineering and Sciences, Tecnol\'{o}gico de Monterrey, Quer\'{e}taro 76130,  Mexico}\\
  \noindent\textit{\scriptsize $^{3}$ SQC and OR Unit, Indian Statistical Institute, Kolkata 700108, India} \\
  \noindent\textit{\scriptsize $^4$ Department of Mathematics, Indian Institute of Technology, Patna 801106, India }\\
 \end{center}

\begin{abstract}
This article describes the inferential procedures and Bayesian optimal life-testing issues under Type-II unified hybrid censoring scheme. First, the  explicit expressions of expected number of failures, expected duration of testing and Fisher information matrix for the unknown parameters of the underlying lifetime model are derived. Then, using these quantities, the Bayesian optimal life-testing plans are computed in subsequent section. A cost constraint D-optimal optimization problem has been formulated and the corresponding solution algorithm is provided to obtain optimal plans. Computational procedures are illustrated through numerical examples.   
\end{abstract}


\section{Introduction} \label{sec1}
\paragraph{}
The two fundamental censoring schemes widely used for conducting life-tests are Type-I and Type-II censoring schemes. Under Type-I censoring, lifetimes are observed only up to a pre-specified time point, say $T_{1}$. In Type-II censoring, life-test is terminated when a pre-specified number of units $r (\leq n)$  have failed where, $n$ is the initial number of units placed on testing. \cite{Epstein_1954} introduced hybrid censoring scheme which is a  combination of both Type-I and Type-II censoring schemes. In Type-I hybrid censoring scheme, for pre-fixed $r$ and $T_1$, the test is conducted upto the time of $r$th failure or time point $T_{1}$, whichever occurs first. \cite{Childs_2003} proposed the Type-II hybrid censoring scheme, in which a life-test is terminated at the time of $r$th failure or at time $T_1$, whichever occurs later. In recent years, generalizing these schemes became an interesting aspects to the researchers. For instances, \cite{Chandrasekar_2004} introduced generalized Type-I and Type-II hybrid censoring schemes, and later on \cite{Balakrishnan_2008} introduced Type-I and Type-II unified hybrid censoring schemes. In this article, we consider Type-II unified hybrid censoring scheme (Type-II UHCS) which is the generalization of generalized Type-I and Type-II hybrid censoring schemes. The Type-II UHCS can be described as follows. The testing starts with $n$ units and alongside two integers $l, r\in \{1, 2,\cdots, n\}$ and two time points $T_1, T_2 \in (0, \infty)$ are chosen such that $l<r$ and $T_1<T_2$. If the $r$th failure occurs before time $T_1$, terminate the test at $T_1$. If the $l$th failure occurs before $T_1$ and $r$th failure occurs between $T_1$ and $T_2$, terminate the test at $r$th failure time. If the $l$th failure occurs before $T_1$ and $r$th failure occurs after $T_2$, terminate the test at $T_2$. If the $l$th failure occurs after $T_1$ and $r$th failure occurs before $T_2$, terminate the experiment at  $r$th failure time. If the $l$th failure occurs after $T_1$ and $r$th failure occurs after $T_2$, terminate the test at $T_2$. Finally, if the $l$th failure occurs after time $T_2$, terminate the experiment at $l$th failure time. The advantage of Type-II UHCS is that it ensures at least $l$  failures and the maximum test duration is $T_2$. {\color{blue}If we consider $l=0$, then Type-II UHCS reduces to Type-II GHCS. And if We don't fix any upper bound of censoring time i.e. if $T_2 -> \infty$, then Type-II UHCS reduces to Type-I GHCS.} A schematic representation of Type-II UHCS is given in Figure 1.

The aim of this paper is two-fold. Firstly, we derived the explicit expressions of expected number of failures and expected duration of testing under Type-II unified hybrid censoring scheme. We also derived the analytical expression of the Fisher information matrix about the unknown parameter(s) of the lifetime distribution under Type-II unified hybrid censoring data. Secondly, the procedures of computing Bayesian optimal life-testing plans under Type-II UHCS are discussed. \cite{Zahar_1996} discussed Bayesian life-testing plans using information based criterion under Type-I censoring scheme. \cite{Zhang_2005} described Bayesian life-test planning for the Weibull distribution with given shape parameter under Type II censoring scheme. \cite{Hong_2015} discussed Bayesian life-test planning under Type-I censoring for log-location-Scale family of distributions. \cite{Ritwik_2018} discussed Bayesian design of life-testing plans under hybrid censoring scheme. A Fisher information based Bayesian design criterion proposed by \cite{Roy_2017} under progressive Type-I interval censoring scheme. As per our best knowledge is concerned, no work has been done  
\begin{figure}[htbp]
	\tiny
	\begin{tikzpicture}[scale=0.7]
    \draw[thick] (0,0) -- (6,0);
	\draw[thick,dotted] (6,0) -- (8,0);
	\draw[thick] (8,0) -- (18,0);
	\draw [fill] (0, 0) circle [radius=0.1];
	\node[left] at (0,0) {\tiny{\textbf{Start at 0}}};
	\draw [fill] (1, 0) circle [radius=0.1];
	\draw[->] (1,0) -- (2,2);
	\node[below] at (1,0) {\tiny{$X_{1:n}$}};
	\node[right] at (2,2) {\tiny{1st failure}};
	\draw [fill] (4, 0) circle [radius=0.1];
	\draw[->] (4,0) -- (5,2);
	\node[below] at (4,0) {\tiny{$X_{2:n}$}};
	\node[right] at (5,2) {\tiny{2nd failure}};
    \draw [fill] (9, 0) circle [radius=0.1];
	\draw[->](9,0) -- (10,2);
	\node[below] at (9,0) {\tiny{$X_{l:n}$}};
	\node[right] at (10,2) {\tiny{lth failure}};
	\draw [fill] (12, 0) circle [radius=0.1];
	\draw[->](12,0) -- (13,2);
	\node[below] at (12,0) {\tiny{$X_{r:n}$}};
	\node[right] at (13,2) {\tiny{rth failure}};
	\draw [fill] (14, 0) circle [radius=0.1];
	%\draw[->](14,0) -- (15,2);
	\node[below] at (14,0) {\tiny{$T_1$}};
	%\node[right] at (15,2) {\tiny{$r$th Failure}};
	\node[below] at (14, -0.5) {\tiny{\textbf{Stop at $T_1$}}};
	\draw [fill] (16, 0) circle [radius=0.1];
	%\draw[->](14,0) -- (15,2);
	\node[below] at (16,0) {\tiny{$T_2$}};
	%\node[right] at (15,2) {\tiny{$r$th Failure}};
	%\node[below] at (16, -0.5) {\tiny{\textbf{Stop at $T_1$}}};
	\node[right] at (18, 0) {};
	\node at (7, -1) {\scriptsize{\textbf{Case I}}};
	\draw[thick] (0,-5) -- (6,-5);
	\draw[thick,dotted] (6,-5) -- (8,-5);
	\draw[thick] (8,-5) -- (18,-5);
	\draw [fill] (0,-5) circle [radius=0.1];
	\node[left] at (0,-5) {\tiny{\textbf{Start at 0}}};
	\draw [fill] (1,-5) circle [radius=0.1];
	\draw[->] (1,-5) -- (2,-3);
	\node[below] at (1,-5) {\tiny{$X_{1:n}$}};
	\node[right] at (2,-3) {\tiny{1st failure}};
	\draw [fill] (4,-5) circle [radius=0.1];
	\draw[->] (4,-5) -- (5,-3);
	\node[below] at (4,-5) {\tiny{$X_{2:n}$}};
	\node[right] at (5,-3) {\tiny{2nd failure}};
	\draw [fill] (9,-5) circle [radius=0.1];
	\draw[->] (9,-5) -- (10,-3);
	\node[below] at (9,-5) {\tiny{$X_{l:n}$}};
	\node[right] at (10,-3) {\tiny{lth failure}};
	\draw [fill] (12,-5) circle [radius=0.1];
	%\draw[->] (12,-5) -- (13,-3);
	\node[below] at (12,-5) {\tiny{$T_1$}};
	%	\node[right] at (13,-3) {\tiny{$R_d^*$}};
	\draw [fill] (14,-5) circle [radius=0.1];
	\draw[->] (14,-5) -- (15,-3);
	\node[below] at (14,-5) {\tiny{$X_{r:n}$}};
	\node[right] at (15,-3) {\tiny{rth failure}};
	\node[below] at (14, -5.5) {\tiny{\textbf{Stop at  $X_{r:n}$}}};
	\draw [fill] (16,-5) circle [radius=0.1];
	%\draw[->] (12,-5) -- (13,-3);
	\node[below] at (16,-5) {\tiny{$T_2$}};
	\node[right] at (18, -5) {};
	\node at (7, -6) {\scriptsize{\textbf{Case II}}};
	%% case 3
	\draw[thick] (0,-10) -- (6,-10);
	\draw[thick,dotted] (6,-10) -- (8,-10);
	\draw[thick] (8,-10) -- (18,-10);
	\draw [fill] (0,-10) circle [radius=0.1];
	\node[left] at (0,-10) {\tiny{\textbf{Start at 0}}};
	\draw [fill] (1,-10) circle [radius=0.1];
	\draw[->] (1,-10) -- (2,-8);
	\node[below] at (1,-10) {\tiny{$X_{1:n}$}};
	\node[right] at (2,-8) {\tiny{1st failure}};
	\draw [fill] (4,-10) circle [radius=0.1];
	\draw[->] (4,-10) -- (5,-8);
	\node[below] at (4,-10) {\tiny{$X_{2:n}$}};
	\node[right] at (5,-8) {\tiny{2nd failure}};
	\draw [fill] (9,-10) circle [radius=0.1];
	\draw[->] (9,-10) -- (10,-8);
	\node[below] at (9,-10) {\tiny{$X_{l:n}$}};
	\node[right] at (10,-8) {\tiny{lth failure}};
	\draw [fill] (12,-10) circle [radius=0.1];
	%\draw[->] (12,-5) -- (13,-3);
	\node[below] at (12,-10) {\tiny{$T_1$}};
	%	\node[right] at (13,-3) {\tiny{$R_d^*$}};
	\draw [fill] (14,-10) circle [radius=0.1];
	%\draw[->] (14,-10) -- (15,-8);
	\node[below] at (14,-10) {\tiny{$T_2$}};
	%\node[right] at (15,-8) {\tiny{rth failure}};
	\node[below] at (14, -10.5) {\tiny{\textbf{Stop at  $T_2$}}};
	\draw [fill] (16,-10) circle [radius=0.1];
	%\draw[->] (12,-5) -- (13,-3);
	\draw[->] (16,-10) -- (17,-8);
	\node[below] at (16,-10) {\tiny{$X_{r:n}$}};
	\node[right] at (17,-8) {\tiny{rth failure}};
	\node[right] at (18, -10) {};
	\node at (7, -11) {\scriptsize{\textbf{Case III}}};
	\end{tikzpicture}
	%\caption{Schematic representation of Type-II UHCS.}
	%\label{HybFig1}	
\end{figure}  
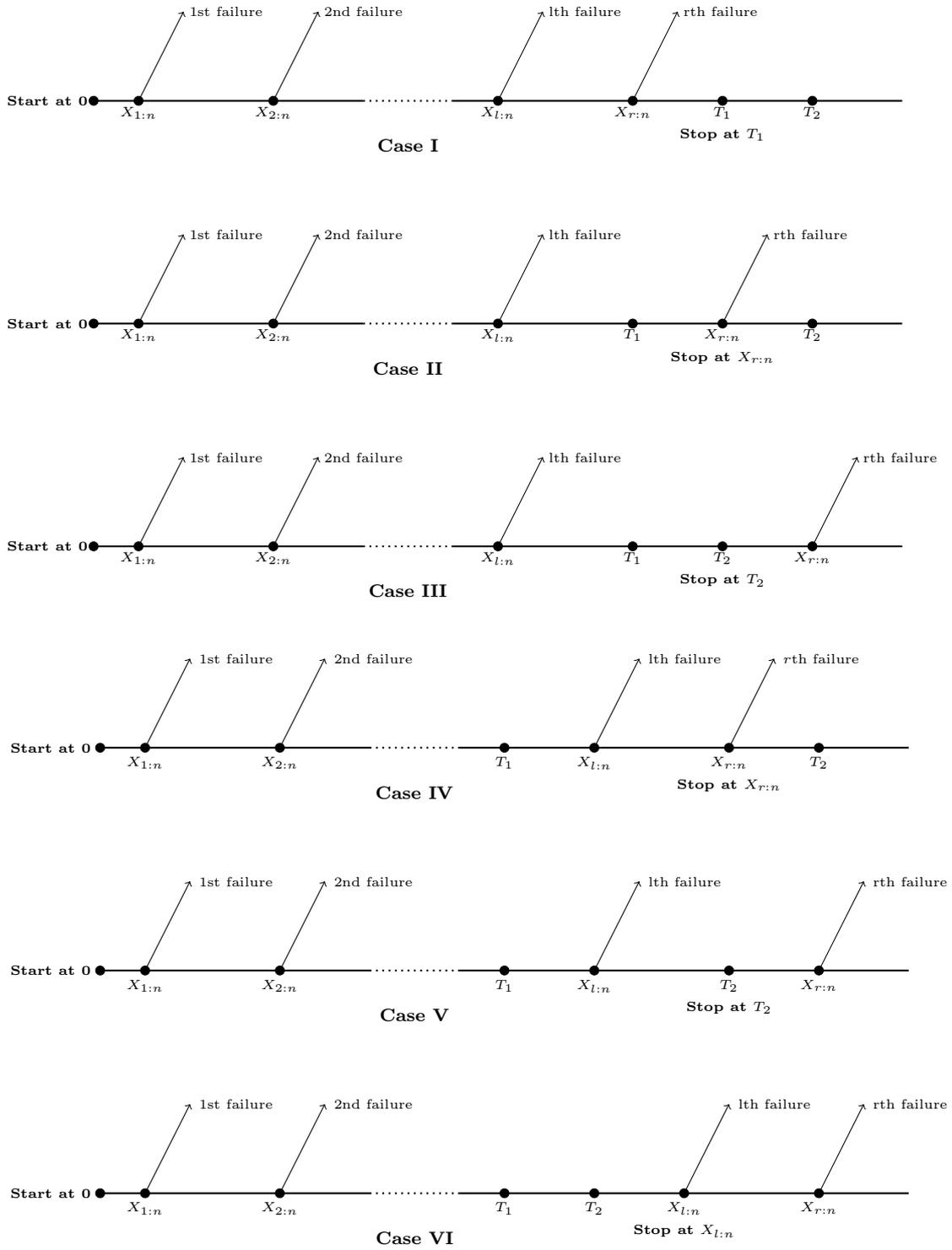
\begin{figure}[htbp]
	\begin{tikzpicture}[scale=0.7]
	
	\draw[thick] (0,0) -- (6,0);
	\draw[thick,dotted] (6,0) -- (8,0);
	\draw[thick] (8,0) -- (18,0);
	
	\draw [fill] (0, 0) circle [radius=0.1];
	\node[left] at (0,0) {\tiny{\textbf{Start at 0}}};
	\draw [fill] (1, 0) circle [radius=0.1];
	\draw[->] (1,0) -- (2,2);
	\node[below] at (1,0) {\tiny{$X_{1:n}$}};
	\node[right] at (2,2) {\tiny{1st failure}};
	
	\draw [fill] (4, 0) circle [radius=0.1];
	\draw[->] (4,0) -- (5,2);
	\node[below] at (4,0) {\tiny{$X_{2:n}$}};
	\node[right] at (5,2) {\tiny{2nd failure}};

	\draw [fill] (9, 0) circle [radius=0.1];
	%\draw[->](9,0) -- (10,2);
	\node[below] at (9,0) {\tiny{$T_1$}};
	%\node[right] at (10,2) {\tiny{kth failure}};
	
	\draw [fill] (11, 0) circle [radius=0.1];
	\draw[->](11,0) -- (12,2);
	\node[below] at (11,0) {\tiny{$X_{l:n}$}};
	\node[right] at (12,2) {\tiny{lth failure}};
	
	\draw [fill] (14, 0) circle [radius=0.1];
	\draw[->](14,0) -- (15,2);
	\node[below] at (14,0) {\tiny{$X_{r:n}$}};
	\node[right] at (15,2) {\tiny{$r$th failure}};
	\node[below] at (14, -0.5) {\tiny{\textbf{Stop at $X_{r:n}$}}};

	\draw [fill] (16, 0) circle [radius=0.1];
	%\draw[->](14,0) -- (15,2);
	\node[below] at (16,0) {\tiny{$T_2$}};
	%\node[right] at (15,2) {\tiny{$r$th Failure}};
	%\node[below] at (16, -0.5) {\tiny{\textbf{Stop at $T_1$}}};
	
	\node[right] at (18, 0) {};
	
	\node at (7, -1) {\scriptsize{\textbf{Case IV}}};
	
	\draw[thick] (0,-5) -- (6,-5);
	\draw[thick,dotted] (6,-5) -- (8,-5);
	\draw[thick] (8,-5) -- (18,-5);
	
	\draw [fill] (0,-5) circle [radius=0.1];
	\node[left] at (0,-5) {\tiny{\textbf{Start at 0}}};
	
	\draw [fill] (1,-5) circle [radius=0.1];
	\draw[->] (1,-5) -- (2,-3);
	\node[below] at (1,-5) {\tiny{$X_{1:n}$}};
	\node[right] at (2,-3) {\tiny{1st failure}};
	
	\draw [fill] (4,-5) circle [radius=0.1];
	\draw[->] (4,-5) -- (5,-3);
	\node[below] at (4,-5) {\tiny{$X_{2:n}$}};
	\node[right] at (5,-3) {\tiny{2nd failure}};
	
	\draw [fill] (9,-5) circle [radius=0.1];
	%\draw[->] (9,-5) -- (10,-3);
	\node[below] at (9,-5) {\tiny{$T_1$}};
	%\node[right] at (10,-3) {\tiny{kth failure}};
	
	\draw [fill] (11,-5) circle [radius=0.1];
	\draw[->] (11,-5) -- (12,-3);
	\node[below] at (11,-5) {\tiny{$X_{l:n}$}};
	\node[right] at (12,-3) {\tiny{lth failure}};

	\draw [fill] (14,-5) circle [radius=0.1];
	%\draw[->] (14,-5) -- (15,-3);
	\node[below] at (14,-5) {\tiny{$T_2$}};
	%\node[right] at (15,-3) {\tiny{rth failure}};
	\node[below] at (14, -5.5) {\tiny{\textbf{Stop at  $T_2$}}};
	
	\draw [fill] (16,-5) circle [radius=0.1];
	\draw[->] (16,-5) -- (17,-3);
	\node[below] at (16,-5) {\tiny{$X_{r:n}$}};
	\node[right] at (17,-3) {\tiny{rth failure}};
	\node[right] at (18, -5) {};
	
	\node at (7, -6) {\scriptsize{\textbf{Case V}}};
	
	%% case 3
	
	\draw[thick] (0,-10) -- (6,-10);
	\draw[thick,dotted] (6,-10) -- (8,-10);
	\draw[thick] (8,-10) -- (18,-10);
	
	\draw [fill] (0,-10) circle [radius=0.1];
	\node[left] at (0,-10) {\tiny{\textbf{Start at 0}}};
	
	\draw [fill] (1,-10) circle [radius=0.1];
	\draw[->] (1,-10) -- (2,-8);
	\node[below] at (1,-10) {\tiny{$X_{1:n}$}};
	\node[right] at (2,-8) {\tiny{1st failure}};
	
	\draw [fill] (4,-10) circle [radius=0.1];
	\draw[->] (4,-10) -- (5,-8);
	\node[below] at (4,-10) {\tiny{$X_{2:n}$}};
	\node[right] at (5,-8) {\tiny{2nd failure}};
	
	\draw [fill] (9,-10) circle [radius=0.1];
	%\draw[->] (9,-10) -- (10,-8);
	\node[below] at (9,-10) {\tiny{$T_1$}};
	%\node[right] at (10,-8) {\tiny{kth failure}};
	
	\draw [fill] (11,-10) circle [radius=0.1];
	%\draw[->] (12,-5) -- (13,-3);
	\node[below] at (11,-10) {\tiny{$T_2$}};
	%	\node[right] at (13,-3) {\tiny{$R_d^*$}};

	\draw [fill] (13,-10) circle [radius=0.1];
	\draw[->] (13,-10) -- (14,-8);
	\node[below] at (13,-10) {\tiny{$X_{l:n}$}};
	\node[right] at (14,-8) {\tiny{lth failure}};
	\node[below] at (13, -10.5) {\tiny{\textbf{Stop at  $X_{l:n}$}}};
	
	\draw [fill] (16,-10) circle [radius=0.1];
	%\draw[->] (12,-5) -- (13,-3);
	\draw[->] (16,-10) -- (17,-8);
	\node[below] at (16,-10) {\tiny{$X_{r:n}$}};
	\node[right] at (17,-8) {\tiny{rth failure}};
	\node[right] at (18, -10) {};
	
	\node at (7, -11) {\scriptsize{\textbf{Case VI}}};
	
	\end{tikzpicture}
	\caption{Schematic representation of Type-II UHCS.}
	\label{UhybFig1}
\end{figure}in the literature on the issues of Bayesian optimal life-testing plan under Type-II UHCS. Moreover, all the previous works were based on sample size determination problem by keeping other design parameters as constant. Nevertheless, in this study, we proposed an optimal design strategy to obtain the optimal values of all design parameters $(n,r,l,T_1,T_2)$ by solving a cost constraint D-optimality criterion. A flowchart representation of the solution algorithm is provided in the relevant section. To reduce the computational complexity, we used the  approximation technique of the Bayesian posterior quantities which are the components of optimality criterion. The proposed method is illustrated through various numerical examples.

The organization of the paper is as follows. In Section 2, some results on Type-II UHCS are derived which will be required to formulate the Bayesian optimality criterion in the subsequent sections. The expressions of Fisher information matrix is also derived in this section. Formulation and computational strategies of the optimal Bayesian life-testing plan are described in Section 3. The solution algorithm and the numerical illustrations are also presented in this section. Finally, some concluding remarks are made in Section 4.

\section{Some results on Type-II UHCS}
\paragraph{}
Suppose that $X_1, X_2.\cdots, X_n$ are the lifetimes of $n$ testing units with common distribution function $F(\cdot; \theta),$ where $\theta$ is parameter(s). The corresponding ordered lifetimes are $X_{1:n}<X_{2:n}<\cdots<X_{n:n}$. Let $D$ and $\xi$ represent the number of failures and the duration of the life-test under a Type-II UHCS, respectively. Then, $(X_{1:n}, X_{2:n},..., X_{D:n}, \xi)$ represents a Type-II UHCS data defined as 
\[ (D,\xi) = \left\{ \begin{array}{lll}
	(D_1,T_1) & \mbox{if $X_{l:n}<X_{r:n}<T_1<T_2, ~~ \mbox{where}~~  D_1= r,r+1,\ldots,n$},\\
	(r,X_{r:n}) & \mbox{if $X_{l:n}<T_1<X_{r:n}<T_2$},\\
	(D_2,T_2) & \mbox{if $X_{l:n}<T_1<T_2<X_{r:n}, ~~ \mbox{where}~~  D_2=l,l+1,\ldots,r-1$},\\
	(r,X_{r:n}) & \mbox{if $T_1< X_{l:n}<X_{r:n}<T_2$},\\
	(D_2,T_2) & \mbox{if $T_1<X_{l:n}<T_2<X_{r:n}, ~~ \mbox{where}~~  D_2=l,l+1,\ldots, r-1$},\\
	(l,X_{l:n}) & \mbox{if $T_1< T_2<X_{l:n}<X_{r:n}$.}
	%	(l,Y_{l:n})& \mbox{if $Y^{l}<Y_{l:n}$}, \\
	%(r,Y_{r:n}) & \mbox{if $Y_{r:n} < Y^{l}$,}
	\end{array} \right. \]	
 Next, we  derive the expressions of the expected numbers of failures $E[D]$, the expected duration of the test $E[\xi]$ and the Fisher information matrix  under Type-II UHCS data, which will be required to construct optimal Bayesian design in Section 3. For notational convenience, we write $\mbox{min}(x,y)$ and $\mbox{max}(x,y)$ as $x \wedge y$ and $x \vee y$, respectively.  
 
 \begin{theorem}
 	The expected number of failures under Type-II UHCS is given as
 	\begin{equation*}
 	E[D]=l+nF(T_1, \theta)+N_{X_{r:n}\wedge T_2}-N_{X_{l:n}\wedge T_2} -N_{X_{r:n}\wedge T_1},
 	\end{equation*}where $N_{X_{r:n}\wedge T_2}, N_{X_{l:n}\wedge T_2}$ and $N_{X_{r:n}\wedge T_1}$ denote the  expected number of failures under the Type-I hybrid censoring schemes $(r, T_2), (l, T_2)$ and $(r, T_1)$, respectively.	
 \end{theorem}
 \begin{proof}
 	Following the notations from Theorem 1 and using the additive rule \citep[see][]{Park_2009}, we can evaluate the expected duration of life-test as
 	\begin{equation}\label{c33}
 	E[D] = N_{S \wedge Q} = N_{S}+N_{Q}-N_{S\vee Q}, 
 	\end{equation}
 	where $N_{S} = E[D| X_{l:n}\vee T_2]$, $N_{Q} = E[D| X_{r:n}\vee T_1]$ and $N_{S\vee Q} = E[D| (X_{l:n}\vee T_2) \vee (X_{r:n}\vee T_1)] = E[D| X_{r:n}\vee T_2]$. Again, on simplification, we have
 	\begin{eqnarray*}
 		N_{S}&=&nF(T_2, \theta) +l-N_{X_{l:n}\wedge T_2}, \\
 		N_{Q}&=&nF(T_1, \theta) +r-N_{X_{r:n}\wedge T_1}, \\
 		N_{S\vee Q}&=&nF(T_2, \theta) +r-N_{X_{r:n}\wedge T_2}.
 	\end{eqnarray*}Hence, replacing the above expressions in (\ref{c33}), the desired result follows as
 	\begin{equation*}
 	E[D]=l+nF(T_1, \theta)+N_{X_{r:n}\wedge T_2}-N_{X_{l:n}\wedge T_2} -N_{X_{r:n}\wedge T_1}.
 	\end{equation*}
 \end{proof}
 \begin{rmk}
 	The expressions for $N_{X_{r:n}\wedge T_2}$, $N_{X_{l:n}\wedge T_2}$ and $N_{X_{r:n}\wedge T_1} $ are given by
 	\begin{eqnarray*}
 		N_{X_{r:n}\wedge T_2}&=&\sum_{i=1}^{r}F_{i:n}(T_2, \theta) dx,\\
 		N_{X_{l:n}\wedge T_2}&=& \sum_{i=1}^{l}F_{i:n}(T_2, \theta) dx, \\
 		N_{X_{r:n}\wedge T_1}&=&\sum_{i=1}^{r}F_{i:n}(T_1, \theta) dx,
 	\end{eqnarray*}where $F_{i:n}(\cdot, \theta)$ is the distribution function of $X_{i:n}$.
 \end{rmk}

Now we derive the expression of $E(\xi)$. Note that $\xi = (x_{l:n} \vee T_2) \wedge (x_{r:n} \vee T_1)$. Then we have the following result. 
 
\begin{theorem}
	
 The expected duration under Type-II UHCS is given as
	\begin{equation}\nonumber
		E[\xi]
		=  E[X_{l:n}]+T_1+C_{X_{r:n}\wedge T_2}-C_{X_{l:n}\wedge T_2} -C_{X_{r:n}\wedge T_1}
	\end{equation}
where $C_{X_{r:n}\wedge T_2}$, $C_{X_{l:n}\wedge T_2}$ and $C_{X_{r:n}\wedge T_1}$ denote the expected duration under Type-I hybrid censoring schemes $(r,T_2)$, $(l,T_2)$ and $(r,T_1)$, respectively. 

\end{theorem}
\begin{proof}
Suppose $S=X_{l:n}\vee T_2$ and $Q=X_{r:n}\vee T_1$ are two random variables representing two termination times. Now, by using the additive rule \citep[see][]{Park_2009}, we can evaluate the expected duration of time  as
\begin{equation}\label{c11}
	E[\xi] = C_{S \wedge Q}=C_{S}+C_{Q}-C_{S\vee Q} 
\end{equation}
where $C_{S}=E[X_{l:n}\vee T_2]$, $C_{Q}=E[X_{r:n}\vee T_1]$ and 
$C_{S \vee Q}=E[(X_{l:n}\vee T_2) \vee (X_{r:n}\vee T_1)] = E[X_{r:n}\vee T_2]$. Further simplifications of the quantities $C_{S}$, $C_{Q}$ and $C_{S \vee Q}$ gives
\begin{eqnarray*}
C_{S}&=&T_2+E[X_{l:n}]-C_{X_{l:n}\wedge T_2}, \\
C_{Q}&=&T_1+E[X_{r:n}]-C_{X_{r:n}\wedge T_1}, \\
C_{S\vee Q}&=&T_2+E[X_{r:n}]-C_{X_{r:n}\wedge T_2}.
\end{eqnarray*}Hence, replacing the above expressions in (\ref{c11}), the desired result follows as
\begin{equation*}
	E[\xi]
	=  E[X_{l:n}]+T_1+C_{X_{r:n}\wedge T_2}-C_{X_{l:n}\wedge T_2} -C_{X_{r:n}\wedge T_1}.
\end{equation*}
\end{proof}
\begin{rmk}
 The expressions for $C_{X_{r:n}\wedge T_2}$, $C_{X_{l:n}\wedge T_2}$ and $C_{X_{r:n}\wedge T_1} $ are given by
 \begin{eqnarray*}
C_{X_{r:n}\wedge T_2}&=&\int_{0}^{T_2}(1-F_{r:n}(x, \theta)) dx,\\
C_{X_{l:n}\wedge T_2}&=& \int_{0}^{T_2}(1-F_{l:n}(x, \theta)) dx,\\
C_{X_{r:n}\wedge T_1}&=&\int_{0}^{T_1}(1-F_{r:n}(x, \theta)) dx. 
.
\end{eqnarray*} 
\end{rmk}

Next, we derive Fisher information matrix under Type-II UHCS. The likelihood function based on Type-II UHCS data is given by  	
\begin{eqnarray*}	
	L(\mu,\tau | x_{1:n}, x_{2:n},..., x_{d_0:n}, \xi) \propto  \prod_{i=1}^{d_0} f_{X}(x_{i:n}; \mu,\tau)(1-F_X(\xi_0; \mu,\tau))^{n-d_0}, 
\end{eqnarray*}	
$x_{i:n}$, $d_0$ and $\xi_0$ denote the observed values of $X_{i:n}$, $D$ and $\xi$, respectively. We present the following two lemmas which will be used in Theorem 3 to derive the expression of Fisher information matrix for $\theta$ under Type-II UHCS.	 
	
\begin{lemma}
	\label{faa}
	The Fisher information about $\theta$ in the Type-II censored data for pre-specified number of failures $l$, $1\leq l \leq n$, is
	\begin{eqnarray*}	
		I_{1\ldots l:n}(\theta)= \int_{0}^{\infty}\bigg \langle \frac{\partial}{\partial \theta}\ln h(x; \theta) \bigg \rangle \sum_{i=1}^{l}f_{i:n}(x; \theta)\, dx,
	\end{eqnarray*}where $h(x; \theta)$ is the hazard function of $X$, $f_{i:n}(x; \theta)$ is the density of $X_{i:n}$, $(\partial/\partial\theta)\ln h(x; \theta)$ is the vector $((\partial/\partial\mu)\ln h(x; \theta), (\partial/\partial\tau)\ln h(x; \theta))^{'}$ and $\langle A \rangle$ is defined as the matrix $A.A^{'},$ for $A\in \mathbb{R}^2$.
\end{lemma}
\begin{proof}
	See Lemma 2.1 in \cite{Park_2008}.
\end{proof}

\begin{lemma}
	\label{lem2fish}
	The Fisher information about $\theta$  in the Type-I HCS data corresponding to $(r, T_1)$ is given by 
	\begin{eqnarray*}	
		I_{X_{r:n}\wedge T_1}(\theta)= \int_{0}^{T_1}\bigg \langle \frac{\partial}{\partial \theta}\ln h(x;\theta) \bigg \rangle \sum_{i=1}^{r}f_{i:n}(x;\theta)\, dx.
	\end{eqnarray*}
\end{lemma}
\begin{proof}
	See Lemma 2.1 in \cite{Park_2008}.
\end{proof}

\begin{theorem}
	%\label{FI1.th}
	The Fisher information about $\theta$ under Type-II UHCS is given by
	\begin{eqnarray*}	
		I(\theta)=I_{1,\ldots,l:n}(\theta)+ I_{T_1}(\theta)+ I_{X_{r:n}\wedge T_2}(\theta) - I_{X_{l:n}\wedge T_2}(\theta)-I_{X_{r:n}\wedge T_1}(\theta).
	\end{eqnarray*}
\end{theorem}

\begin{proof}
 
 Let us define
 \[ Y_{i} = \left\{ \begin{array}{ll}
              X_{i:n}, & \mbox{for $i=1,2,\cdots,l$},\\
              (X_{i:n}\wedge T_2, \textbf{I}(X_{i:n}\leq T_2)), & \mbox{for $i=l+1,\cdots,r$},\\
              (X_{i:n}\wedge T_1, \textbf{I}(X_{i:n}\leq T_1)) & \mbox{for $i=r+1,\cdots,n$},
                 \end{array}
\right.\]where $\textbf{I}(\cdot)$ is indicator function. Therefore, the Type-II UHCS data can be represented by $$(Y_1,\ldots,Y_{l},Y_{l+1},\ldots,Y_{r},Y_{r+1},\cdots,Y_{n}).$$By the Markov chain property of order statistics, the joint density function of $(Y_1,\cdots,Y_{l},\cdots,Y_{r},\ldots,Y_n)$ can be decomposed as
\begin{eqnarray*}
		&& f_{1,\cdots,l,\cdots,r,\cdots,n:n}(y_1,\cdots,y_l,\cdots,y_{r},\cdots,y_{n};\theta)\\
		&&=f_{1,\cdots,r:n}(y_1,\cdots,y_{r};\theta) \, f_{r+1,\cdots,n|r,\cdots,1:n}(y_{r+1},\cdots,y_{n}|y_r,\cdots,y_1;\theta)\\
		&&= f_{1,\cdots,l:n}(y_1,\cdots,y_{l};\theta) \, f_{l+1,\cdots,r|l:n}(y_{l+1},\cdots,y_{r}|y_l;\theta)f_{r+1,\cdots,n|r:n}(y_{r+1},\cdots,y_{n}|y_r;\theta).
		\label{ent1}
	\end{eqnarray*}	
 Let us denote the Fisher information about $\theta$ corresponding to the joint densities $f_{1,\ldots,l:n}(y_1,\cdots,y_{l}; \theta)$, $f_{l+1,\cdots,r|l:n}(y_{l+1}, \cdots,y_{r}|y_l;\theta)$ and $f_{r+1,\cdots,n|r:n}(y_{r+1}, \cdots,y_{n}|y_r;\theta)$ by 
 $I_{1,\cdots,l:n}(\theta)$, $I_{l+1,\cdots,r \mid l:n}(\theta)$ and $I_{r+1,\cdots,n\mid r:n}(\theta)$, respectively. By using Lemma \ref{faa}, we obtain the expression for  $I_{1,\cdots,l:n}(\theta)$ as
 \begin{equation}
	\label{f11}	
	I_{1,\ldots,l:n}(\theta)=\int_{0}^{\infty}\bigg \langle\frac{\partial}{\partial \theta}\ln h(x;\theta) \bigg \rangle \sum_{i=1}^{l}f_{i:n}(x;\theta)\, dx.
	\end{equation}
 To evaluate the Fisher information of $I_{l+1,\cdots,r \mid l:n}(\theta)$ and $I_{r+1,\cdots,n\mid r:n}(\theta)$, we use the following decomposition
 {\footnotesize{
 \begin{eqnarray}\nonumber
  f_{l+1,\cdots,r \mid l:n}(y_{l+1}, \cdots, y_r|y_l; \theta) &=& \frac{f_{l, l+1, \cdots, r:n}(y_l, y_{l+1}, \cdots, y_r| \theta)}{f_{l}(y_l; \theta)}\\\nonumber
  &=& \frac{f_{l, l+1, \cdots, r:n}(y_l, y_{l+1}, \cdots, y_r; \theta)}{f_{l, l+1, \cdots, r-1:n}(y_l, y_{l+1}, \cdots, y_{r-1}; \theta)}\times \frac{f_{l, l+1, \cdots, r-1:n}(y_l, y_{l+1}, \cdots, y_{r-1}; \theta)}{f_{l:n}(y_l; \theta)}\\\nonumber
  &=& \frac{f_{l, l+1, \cdots, r:n}(y_l, y_{l+1}, \cdots, y_r; \theta)}{f_{l, l+1, \cdots, r-1:n}(y_l, y_{l+1}, \cdots, y_{r-1}; \theta)}\times \frac{f_{l, l+1, \cdots, r-1:n}(y_l, y_{l+1}, \cdots, y_{r-1}; \theta)}{f_{l, l+1, \cdots, r-2:n}(y_l, y_{l+1}, \cdots, y_{r-2}; \theta)}\\\nonumber
  && \times\cdots \frac{f_{l, l+1:n}(y_l, y_{l+1}; \theta)}{f_{l:n}(y_l; \theta)}\\\nonumber
  &=& f_{r|l,\cdots,r-1:n}(y_r| y_l,\cdots, y_{r-1}; \theta) \times f_{r-1|l,\cdots,r-2:n}(y_{r-1}| y_l,\cdots, y_{r-2}; \theta)\\\nonumber
  &&\times \cdots\times f_{l+1|l:n}(y_{l+1}| y_l; \theta)\\\nonumber
  &=& f_{r|r-1:n}(y_r|y_{r-1}; \theta)\times f_{r-1|r-2:n}(y_{r-1}|y_{r-2}; \theta)\times \cdots \times f_{l+1|l:n}(y_{l+1}| y_l; \theta)\\\nonumber
  &=& \prod_{j=l+1}^{r}f_{j|j-1:n}(y_{j}|y_{j-1}; \theta),
 \end{eqnarray}}}and, similarly,
 \begin{eqnarray}\nonumber
  f_{r+1,\ldots,n|r:n}(y_{r+1},\ldots,y_{n}; \theta) &=& \prod_{j=r+1}^{n}f_{j|j-1:n}(y_{j}|y_{j-1}; \theta).
 \end{eqnarray}where $f_{j|j-1:n}(y_{j}|y_{j-1}; \theta)$ represents the conditional density of $Y_{j}$ given $Y_{j-1}=y_{j-1}$ \citep[see][]{Park_2008}. Note that this  also can be interpreted as the density of the first order statistic, say $Z_{1:N_{j-1}}$, of sample size $N_{j-1}=n-j+1$, where $N_{j-1}$ denotes the remaining units on the experiment after occurrence of the $j$th failure, that is, it is the  density of $Z_{1:N_{j-1}}$ truncated at $z \geq y_{j-1}$. Suppose $I_{Z_{1:n_{j-1}} \wedge T_2}(\theta)$ denotes the Fisher information about $\theta$ corresponding to $\{ Z_{1:N_{j-1}} \wedge T_2, {\textbf{I}}(Z_{1:N_{j-1}}\leq T_2)\}$, for $j=l+1,\ldots,r$. Then, the Fisher information corresponding to $f_{j|j-1:n}(y_{j}|y_{j-1}; \theta)$ is given by \citep[see][]{Park_2008}
 \begin{equation}
	\label{park1}
	I_{Z_{1:n_{j-1}} \wedge T_2}(\theta)= \int_{0}^{T_2}\bigg \langle \frac{\partial}{\partial \theta}\ln h(x; \theta) \bigg \rangle f_{j:n}(x; \theta)\, dx.
	\end{equation}  
 Similarly, suppose that $I_{Z_{1:n_{j-1}} \wedge T_1}(\theta)$ denotes the Fisher information about $\theta$ corresponding to $\{ Z_{1:N_{j-1}} \wedge T_1, {\textbf{I}}(Z_{1:N_{j-1}}\leq T_1)\}$, for $j=r+1,\ldots,n$. Then, the Fisher information corresponding to $f_{j|j-1:n}(y_{j}|y_{j-1}; \theta)$ is given as
	\begin{equation}
	\label{park2}
	I_{Z_{1:n_{j-1}} \wedge T_1}(\theta)= \int_{0}^{T_1}\bigg \langle \frac{\partial}{\partial \theta}\ln h(x; \theta) \bigg \rangle f_{j:n}(x; \theta)\, dx.
	\end{equation}   
By using the previous decomposition, the Fisher information corresponding to $f_{l+1,\ldots,r|l:n}(y_{l+1},\ldots,y_{r};\theta)$ and $f_{r+1,\ldots,n|r:n}(y_{r+1},\ldots,y_{n};\theta)$ can be written as    
	\begin{eqnarray}
	\label{park3}
	I_{l+1,\ldots,r\mid l:n}(\theta)&=& \sum_{j=l+1}^{r}  I_{Z_{1:n_{j-1}} \wedge T_2}(\theta), 	\label{park3} \\
	I_{r+1,\ldots,n\mid r:n}(\theta)&=& \sum_{j=r+1}^{n}  I_{Z_{1:n_{j-1}} \wedge T_1}(\theta), 	\label{park4}
	\end{eqnarray}respectively. Therefore, by using (\ref{park1}), (\ref{park3}) can be written as    
\begin{eqnarray}
	\label{f22}
	I_{l+1,\ldots,r\mid l:n}(\theta)&=& \int_{0}^{T_2}\bigg \langle \frac{\partial}{\partial \theta}\ln h(x; \theta) \bigg \rangle \sum_{j=l+1}^{r} f_{j:n}(x; \theta)\, dx  \nonumber\\
	&=& \int_{0}^{T_2}\bigg \langle \frac{\partial}{\partial \theta}\ln h(x; \theta) \bigg\rangle \sum_{j=1}^{r}f_{j:n}(x; \theta)\, dx  \nonumber\\&&  -\int_{0}^{T_2}\bigg \langle \frac{\partial}{\partial \theta}\ln h(x; \theta) \bigg \rangle \sum_{j=1}^{l}f_{j:n}(x;\theta)\, dx  \nonumber\\
	&=& I_{X_{r:n}\wedge T_2}(\theta) - I_{X_{l:n}\wedge T_2}(\theta).
	\end{eqnarray}
The final line follows from the Lemma \ref{lem2fish}. Similarly, by using (\ref{park2}), (\ref{park4}) can be written as
\begin{equation}\label{f33}
I_{r+1,\ldots,n\mid r:n}(\theta) = I_{T_1}(\theta) - I_{X_{r:n}\wedge T_1}(\theta). 
\end{equation}
Therefore, the Fisher information about $\theta$ under Type-II UHCS can be obtained by the equations (\ref{f11}), (\ref{f22}) and (\ref{f33}) as
	\begin{eqnarray}\label{f34}	
		I(\theta)=I_{1,\ldots,l:n}(\theta)+ I_{T_1}(\theta)+ I_{X_{r:n}\wedge T_2}(\theta) - I_{X_{l:n}\wedge T_2}(\theta)-I_{X_{r:n}\wedge T_1}(\theta).
	\end{eqnarray}Hence, the proof follows. 
 
\end{proof}

\begin{rmk}	
	
	The equation (\ref{f34}) represents the Fisher information for one dimensional $\theta$. But we can easily derived Fisher information for vector of parameters $\theta$. The Fisher information for $\theta=(\theta_1,\ldots,\theta_p)$ in Type-II UHCS data is given by
	\begin{eqnarray}		
	\label{mul1}
	I(\theta_1,\ldots,\theta_p)=I_{1,\ldots,l:n}(\theta_1,\ldots,\theta_p)+ I_{T_1}(\theta_1,\ldots,\theta_p)+ I_{X_{r:n}\wedge T_2}(\theta_1,\ldots,\theta_p) \nonumber \\ - I_{X_{l:n}\wedge T_2}(\theta_1,\ldots,\theta_p)-I_{X_{r:n}\wedge T_1}(\theta_1,\ldots,\theta_p)
	\end{eqnarray}
	with 
	\begin{eqnarray*}
			I_{1,\ldots,l:n}(\theta)&=& \int_{0}^{\infty}\bigg \langle\frac{\partial}{\partial \theta}\ln h(x;\theta) \bigg \rangle \sum_{i=1}^{l}f_{i:n}(x;\theta)\, dx., \\		
		I_{X_{m:m:n}\wedge T_1}(\theta_1,\ldots,\theta_p)&=& \int_{0}^{T_1}\bigg\langle\frac{\partial}{\partial \theta}\ln h(x;\theta_1,\ldots,\theta_p) \bigg \rangle \sum_{i=1}^{m}f_{i:m:n}(x;\theta_1,\ldots,\theta_p)\, dx,\\
		I_{X_{l:m:n}\wedge T_1}(\theta_1,\ldots,\theta_p)&=& \int_{0}^{T_1}\bigg\langle\frac{\partial}{\partial \theta}\ln h(x;\theta_1,\ldots,\theta_p) \bigg \rangle \sum_{i=1}^{l}f_{i:m:n}(x;\theta_1,\ldots,\theta_p)\, dx,
	\end{eqnarray*}
	where $\langle A\rangle$ represents the matrix $A.A^{'}$ for $A \in \mathbb{R}^{p}$ and $\frac{\partial}{\partial \theta}\ln h(x;\theta_1,\ldots,\theta_p)$ is the vector $\big ( \frac{\partial}{\partial \theta_1}\ln h(x;\theta_1,\ldots,\theta_p),\ldots, \frac{\partial}{\partial \theta_p}\ln h(x;\theta_1,\ldots,\theta_p)\big)^{'}$.
\end{rmk}

\section{Bayesian optimal life-testing plans}
\label{S3}
In this section, we consider determination of optimal life-testing plans under Type-II UHCS under Bayesian framework. By optimal life-testing plan, we refer to the best choice of the design parameters $(n, r, l, T_{1}, T_2)$ of the Type-II UHCS for which a suitably chosen utility function is optimized. Bayesian design comes into picture when the uncertainty of the parameters of the underlying lifetime distribution can be modeled through a probability distribution, called prior distribution. This is combined with the likelihood function to construct posterior distribution and, thus, posterior is used for inference under Bayesian setups. Most of the existing works in Bayesian optimal design consider a suitable posterior variance measure as utility function. For instances, sample size determination problem in Bayesian accelerated life-testing was discussed by \cite{Polson_book}. Bayesian design under Type-II censoring scheme with Weibull lifetime model was described by \cite{Zhang_2005}. \cite{Zhang_2006} also developed Bayesian design for accelerated life-testing assuming acceleration model as linear in parameters. \cite{Hong_2015} proposed a Bayesian design in which posterior variance criterion was computed by using a large sample approximation technique under Type-II censoring. \cite{Ritwik_2018} extends their idea to hybrid censoring schemes. Some more relevant works on Bayesian design can be found in \cite{Thyregod_1975, Zahar_1996, Biswa_2009}. \\

In this work, a Fisher information based variance criterion is minimized subject to a budget constraint. For computing, we use the method proposed in \cite{Hong_2015} and \cite{Ritwik_2018}. Let us define $\mathcal{D}=(n, r, l, T_{1}, T_2)$. To compute the optimal value for $\mathcal{D}$, the following optimization problem is formulated 
\begin{eqnarray}\label{multi1}
&&\underset{\mathcal{D}}{\text{Maximize}}~  E_{\tiny{\mbox{data}};\mathcal{D}} \left[ \ln [\mbox{Det}(I(\theta))|\mbox{data}]  \right]   \\\nonumber
&&\mbox{Subject to} \\\label{opt2}
&& C_f E_{\tiny{\mbox{data}};\mathcal{D}} [ E(D)|\mbox{data} ]    + C_t E_{\tiny{\mbox{data}};\mathcal{D}} [ E(\xi)|\mbox{data} ] \leq C_b,
\end{eqnarray}
where $C_f$, $C_t$ and $C_b$ are the cost per unit of failed item, the cost per unit of duration of life-testing and a pre-fixed budget amount, respectively, and $\mbox{Det}(I(\theta))$ represents the determinant of the Fisher information matrix $I(\theta)$ defined in Theorem 5. Here, the notation $E_{\tiny{\mbox{data}};\mathcal{D}} [\cdot |\mbox{data}]$ represents the Bayesian structure of the design problem. The quantities in (\ref{multi1}) and (\ref{opt2}) can be computed by using any standard MCMC technique, however, the computation will be time consuming. Instead, we are using the large sample approximation approach used in \cite{Hong_2015} and \cite{Ritwik_2018}. They have shown that the approximation technique significantly reduces the computational complexity without loosing much efficiency. Along with the lines of \cite{Hong_2015} and \cite{Ritwik_2018}, assuming $\pi(\theta)$ as the joint prior distribution for $\theta$, the above optimization problem can be approximated as   
%\begin{eqnarray}\nonumber
%&&\underset{\mathcal{D}}{\text{Maximize}}~    \int_{\theta} \ln [\mbox{Det}(I(\theta))] \pi(\theta) \mbox{d}\theta   \\\nonumber
%&&\mbox{Subject to} \\\nonumber
%&& C_f \int_{\theta} E(D) \pi(\theta) \mbox{d}\theta    + C_t \int_{\theta} E(\xi) \pi(\theta) \mbox{d}\theta \leq C_b.
%\end{eqnarray}

\begin{eqnarray}\label{bopt1}
&&\underset{\mathcal{D}}{\text{Maximize}}~ \psi(\mathcal{D}\mid \theta)    \\\nonumber
&&\mbox{Subject to} \\\nonumber
&& C_f \psi_{\tiny{\mbox{Fail}}}(\mathcal{D}\mid \theta) + C_t \psi_{\tiny{\mbox{Duration}}}(\mathcal{D}\mid \theta) \leq C_b.
\end{eqnarray}
where, $\psi(\mathcal{D}\mid \theta) = \int_{\theta} \ln [\mbox{Det}(I(\theta))] \pi(\theta) \mbox{d}\theta$, $\psi_{\tiny{\mbox{Fail}}}(\mathcal{D}\mid \theta) = \int_{\theta} E(D) \pi(\theta) \mbox{d}\theta$, and, $\psi_{\tiny{\mbox{Duration}}}(\mathcal{D}\mid \theta)=\int_{\theta} E(\xi) \pi(\theta) \mbox{d}\theta$.\\

For the purpose of illustration, it is assumed that the actual lifetime of the testing unit follows a log-normal distribution LN$(\mu, \tau)$ with the probability density function (PDF) and the cumulative distribution function (CDF) given by
\begin{equation}\nonumber
 f_X(x; \mu, \tau) = \sqrt{\frac{\tau}{2 \pi}} x^{-1} e^{-\frac{\tau}{2}(\ln x -\mu)^{2}},\,\, x > 0, \,-\infty < \mu < \infty,\, \tau >0,
\end{equation}
and
\begin{equation}\nonumber
 F_X(x; \mu, \tau) =\Phi \textbf [\sqrt{\tau}\,(\ln x- \mu) \textbf]\,,~~ x>0,
\end{equation}respectively, where $\mu$ and $\tau$ denote unknown parameters of the distribution. Here, $\Phi(\cdot)$ is the CDF of standard normal distribution. The log-normal distribution is quite popular distribution in reliability studies because of the flexibility of its shape. It is also assumed that the joint prior distribution of $\theta = (\mu, \tau)$ follows a normal-gamma distribution with probability density function 

\begin{eqnarray}\nonumber
\pi(\mu,\tau)&=& \frac{b_1^{a_1}}{\Gamma {a_1}}\sqrt{\frac{q_2}{2\pi}}\,\,\tau^{a_1-\frac{1}{2}}\,\,e^{-\frac{q_2\tau}{2}(\mu-p_2)^2-b_1\tau}\\\nonumber
&=& \frac{b_1^{a_1}}{\Gamma {a_1}} \,\tau^{a_1-1}\,\,e^{-b_1\tau} \times \frac{1}{\sqrt{2\pi}}\frac{1}{\sqrt{\frac{1}{\tau q_2}}}\,e^{-\frac{(\mu-p_2)^2}{2\frac{1}{\tau q_2}}}\\
&=&\pi(\tau) \times \pi(\mu \mid \tau),
\end{eqnarray}
where, $\pi(\tau) \sim Gamma(a_1,b_1)$ and $\pi(\mu \mid \tau) \sim N_{\mu \mid \tau}(p_2,q_2)$. The hyper parameters $a_1,b_1,p_2$ and $q_2$  reflect prior knowledge about unknown parameters of interest, where $a_1,b_1 >0, q_2>0$ and $-\infty<p_2<\infty$. Now the optimization problem (\ref{bopt1}) is a mixed integer non-linear programming problem in $(n, r, l, T_1 , T_2)$. Algorithm 1, a complete search technique, is proposed for solving (\ref{bopt1}). For all possible combination of $n,r$, and, $l$,  we  solve the constraint optimization problem (\ref{bopt1}) and find the optimal $(n^*,r^*,l^*,T_1^*,T_2^*)$. For solving the constraint optimization problem (\ref{bopt1}), we use \textit{nloptr} package in R-language. 
\begin{algorithm}[!h]
	\label{algo*}
	\small
	\SetAlgoLined
	Fix the values of hyper parameters $a_1,b_1,p_2$, and, $q_2$.\\
	Choose a sufficiently large number $N$.\\
	\For{$i=1,2,\cdots,N$}{
		Generate, $\tau_{i} \sim Gamma(a_1,b_1)$ and $\mu_{i} \sim N_{\mu \mid \tau}(p_2,q_2)$
	}
	By Monte Carlo approximation, $\psi(\mathcal{D}\mid \theta)$,  $\psi_{\tiny{\mbox{Fail}}}(\mathcal{D}\mid \theta)$, and $\psi_{\tiny{\mbox{Duration}}}(\mathcal{D}\mid \theta)$ are approximated as  $\psi^*(\mathcal{D}) = \frac{1}{N}\sum_{i=1}^{N} \ln [\mbox{Det}(I(\theta_i))]$, $\psi_{\tiny{\mbox{Fail}}}^*(\mathcal{D})=\frac{1}{N}\sum_{i=1}^{N}\psi_{\tiny{\mbox{Fail}}}(\mathcal{D}\mid \theta_i)$, and, $\psi_{\tiny{\mbox{Duration}}}^*(\mathcal{D})=\frac{1}{N}\sum_{i=1}^{N} \psi_{\tiny{\mbox{Duration}}}(\mathcal{D}\mid \theta_i)$, $\theta_i = (\mu_i, \tau_{i})$. \\
	
	Fix the values of cost parameters $C_f,C_t$ and $C_b$.\\
	Choose a large value for $n$, say $n^*$.\\	
	
	%Choose, $n=n^*$, where $n^*=\mbox{max}\lbrace n:C_nn\leq C_b \rbrace$,\\
	\For{$n=1,2,\cdots,n^*$}{
		\For{$r=1,2,\cdots,n$}{
			Fix, $l = \lceil(r/2) \rceil$ \\
			Solve the constrained optimization problem (\ref{bopt1}).
			%$ \mathcal{C}(T_1)=C_{f}\,\psi_{Fail}^*(\mathcal{D})+C_{t}\psi_{Duration}^*(\mathcal{D}) = \, C_{b}$ for $T_1$\\
		}
		
		%Compute, $\psi^*(\mathcal{D})$ for all $(r,T_1)$
	}
	Select $(n^*,r^*,{T_1}^*,T_2^*)$ as optimal solution for which $\psi^*(\mathcal{D})$ is maximum.
	\caption{\small{}}
\end{algorithm}

\subsection{ Numerical illustration of optimal plans} 

 For the purpose of illustration, we consider two priors. The priors are determined by taking different means and variances of $\mu$ and $\tau$.
 \begin{description}
  \item[Prior 1:] $E(\mu)=-0.5$, var$(\mu)=0.5$, $E(\tau)=1.5$ and var$(\tau)=1$, we compute the corresponding hyper parameter values as $a_1=2.25; b_1=1.5; p_2=-0.5; q_2=2.4$.
  \item[Prior 2:] $E(\mu)=0.01$,var$(\mu)=0.05$, $E(\tau)=0.5$ and var$(\tau)=0.05$, we compute the corresponding hyper parameter values as $a_1=5; b_1=10; p_2=0.01; q_2=50$.
 \end{description}Note that, for the shake of illustrations, we consider only two priors but, in practice, one can consider other priors by taking suitable combinations of mean and variance of the distribution of $\mu$ and $\tau$. By keeping fix the cost $(C_f, C_t) = (10, 15)$, we compute the optimal schemes for different budget cost $C_b$. Results are reported in Tables 1-4. In Tables 1 and 2, we provide optimal $(r,T_1,T_2)$, $l = \lceil r*0.5 \rceil$, for fix values of $n$ under prior 1 and prior 2, respectively. It is observed that when budget $C_b$ increases, $T_2$ and $r$ increase, also the optimal value. This is because, intuitively, increasing budget allows increasing duration of testing resulting more failures to be observed. It is also observed that when $n$ increases, $(T_1,T_2), r$ and the  optimal value increase, as expected. For fixed budget $C_b$ and sample size $n$, Table-2 (corresponding to prior 2) has higher optimal values in comparison with Table 1 (corresponding to prior 1). Intuitively, this is because the variances of the hyper parameters in Prior 2 is lesser than that of Prior 1 which signifies less uncertainty (in other words, more relevant information about the unknown parameters) resulting higher Fisher information. Instead of fixing $n$, we can also find $n$ optimally for the same cost values. The results are reported in Tables 3 and 4. From Tables 3 and 4, it is observed that when budget increases, all the design parameters values increase, as expected.

 \begin{table}[htbp]
 	\footnotesize
 	\label{table:1}
 	\centering
 	\caption{ Optimal Bayesian life-testing plans under Type-II UHCS with Prior 1 for fix $n$} %for prior 1given $n$ and $l$ for the hyper parameters  %$a_1=2.25; b_1=1.5; p_2=-0.5; q_2=2.4$, where $E(\mu)=-0.5$,var$(\mu)$=0.5, $E(\tau)=1.5$ and var$(\tau)$=1}
 	\begin{center}
 		\begin{tabular}{ccccccc}\toprule		
 			
 			$n$ &$(C_f, C_t)$& $C_b$ & $(T_1,T_2)$ & $r$ & $l$ &$D$-optimal \\
 			\toprule
 			%	20&(10, 15)& 110 & () &  &  \\
 			20& \multirow{15}{*}{(10, 15)} & 150 & (0.7044, 1.4088)  &13 & 7  &  4.4623\\
 			&        & 180 &  (1.1133,  2.2267) & 15 &8& 4.7102  \\ \vspace{0.2cm}
 			&        & 200 &  (1.3148, 2.6296) & 17 &9 & 4.8408 \\
 			30&        & 180 &  (0.4891, 0.9782)& 17 &  9& 4.9483  \\
 			&        & 200 &  (0.5863, 1.1726)& 19 & 10& 5.1583  \\ \vspace{0.2cm}
 			&        & 250 &  (0.8656, 1.7312) & 23 &12& 5.5119  \\
 			40&        & 250 &  (0.5549, 1.1099)& 23&12 &    5.6165   \\
 			&        & 280 &  (0.6783, 1.3567)& 26 & 13& 5.7941    \\   \vspace{0.2cm}
 			&        & 300 &  (0.7281, 1.4562)& 29  & 15& 5.9126   \\
 			50&        & 300 &  (0.4979, 0.9958)& 28  & 14& 6.0064 \\
 			&        & 350 &  (0.7088, 1.4176)&33  &  17& 6.2299 \\     \vspace{0.2cm}
 			&        & 400 &  (0.9171, 1.8343)&37  & 19& 6.4950 \\
 			60&        & 400 &  (0.6552, 1.3104)&37  & 19& 6.6279 \\
 			&        & 450 &  (0.8301, 1.6603)&42 & 21&  6.7682      \\
 			&        & 500 &  (1.1184, 2.2369)& 47 &  24& 6.9497      \\
 			%30&        & 180 &  ()                  &     &          \\
 			%  &        & 200 &  ()                  &     &          \\
 			%  &        & 250 &  ()                  &     &          \\

 			\bottomrule
 		\end{tabular}
 	\end{center}
 \end{table}
 
 \begin{table}[htbp]
 	\footnotesize
 	\label{table:2}
 	\centering
 	\caption{  Optimal Bayesian life-testing plans under Type-II UHCS with Prior 2 for fix $n$} % for the hyper parameters  $a_1=5; b_1=10; p_2=0.01; q_2=50$.$E(\mu)=0.01$,var$(\mu)$=0.05, $E(\tau)=0.5$ and var$(\tau)$=0.05}
 	\begin{center}
 		\begin{tabular}{ccccccc}\toprule		
 			
 			$n$ &$(C_f, C_t)$& $C_b$ & $(T_1,T_2)$ & $r$&l & $D$-optimal \\
 			\toprule

 			%20&(10, 15)& 110 & (0.6007, 1.2014) & 10 & 4.9550 \\
 			20&\multirow{15}{*}{(10, 15)}  & 150 & (0.8816, 1.7632)  &  14 & 7&5.3539 \\
 			&        & 180 &  (1.1976, 2.3952) & 18 &9& 5.5517  \\ \vspace{0.2cm}
 			%&        & 200& (1.5229, 3.0459) & 20   &  5.6795 \\
 			&        & 200 &  (1.5229, 3.0459) & 20 &10& 5.6795  \\
 			30&        & 180 &  (0.7134, 1.4269)& 17 &9& 5.9033  \\
 			&        & 200 &  (0.8302, 1.6605)& 19 &10& 6.1167  \\ \vspace{0.2cm}
 			&        & 250 &  (1.2361, 2.4723) & 24 &12& 6.3872  \\
 			40&        & 250 & (0.9430, 1.8860)& 23 & 12& 6.6131     \\
 			&        & 280 &  (0.9811, 1.9622)& 27 & 14&  6.7884 \\   \vspace{0.2cm}
 			&        & 300 &  (1.0920, 2.1841)&29 &  15& 6.8386 \\
 			50&        & 300 &  (0.9102, 1.8205)& 28 & 14& 6.9971       \\
 			&        & 350 &  (1.1405, 2.2810)  & 33  &17& 7.2564  \\     \vspace{0.2cm}
 			&        & 400 &  (1.3554, 2.7108)  & 39  &20& 7.4236  \\
 			60&        & 400 &  (1.0916, 2.1833) & 38  &19& 7.5210  \\
 			&        & 450 &  (1.3157, 2.6315) &  43 &22&  7.7072       \\
 			&        & 500 &  (1.5655, 3.1310) &  49  & 25& 7.8346       \\
 			%30&        & 180 &  ()                  &     &          \\
 			%  &        & 200 &  ()                  &     &          \\
 			%  &        & 250 &  ()                  &     &          \\

 			\bottomrule
 		\end{tabular}
 	\end{center}
 \end{table}

 \begin{table}[htbp]
 	\footnotesize
 	\label{table:3a}
 	\centering
 	\caption{  Optimal Bayesian life-testing plans $(n,r,l,T_1,T_2)$ under Type-II UHCS with Prior 1}
 		
 		% Optimum Bayesian design under Type-II UHCS for the hyper parameters $a_1=2.25; b_1=1.5; p_2=-0.5; q_2=2.4$.}
 	\begin{center}
 		\begin{tabular}{cccc}\toprule		
 			
 			$(C_f, C_t)$& $C_b$ & $(n,r,l,T_1,T_2)$  & $D$-optimal \\
 			\toprule	
 			(10,15)&80   & (10,7, 0.5662912 , 1.7648)    & 3.128769 \\ 
 			&300   & (20, 15, 8, 1.1004, 2.2008)    & 4.8997 \\ 
 			& 350  &  (23, 19, 10, 1.0389, 2.0778)   & 4.9206 \\ 
 			& 400  &  (26, 22, 11, 1.1150, 2.2300)   & 5.3827 \\ 
 			& 450  & (29, 25, 13, 1.1575, 2.3150)    &  5.5152\\ 
 			& 500  &  (32, 27, 14, 1.4374, 2.8748)   &  5.6757\\

 			\bottomrule
 		\end{tabular}
 	\end{center}
 \end{table}

 \begin{table}[htbp]
	\footnotesize
	\label{table:3b}
	\centering
	\caption{  Optimal Bayesian life-testing plans $(n,r,l,T_1,T_2)$ under Type-II UHCS with Prior a1 = 2.25; b1 = 1.50; p2 = -0.50; q2 = 12.00 with mean and variance 0.01,0.05,0.5,0.05}
	
	% Optimum Bayesian design under Type-II UHCS for the hyper parameters $a_1=2.25; b_1=1.5; p_2=-0.5; q_2=2.4$.}
	\begin{center}
		\begin{tabular}{cccc}\toprule		
			
			$(C_f, C_t)$& $C_b$ & $(n,r,l,T_1,T_2)$  & $D$-optimal \\
			\toprule	
			(10,15)&80  & (10,7,7, 0.8824, 1.7648)    & 4.4485 \\ 
			&300   & (20, 15, 8, 1.1004, 2.2008)    & 4.8997 \\ 
			& 350  &  (23, 19, 10, 1.0389, 2.0778)   & 4.9206 \\ 
			& 400  &  (26, 22, 11, 1.1150, 2.2300)   & 5.3827 \\ 
			& 450  & (29, 25, 13, 1.1575, 2.3150)    &  5.5152\\ 
			& 500  &  (32, 27, 14, 1.4374, 2.8748)   &  5.6757\\

			\bottomrule
		\end{tabular}
	\end{center}
\end{table}

 \begin{table}[htbp]
 	\footnotesize
 	\label{table:4}
 	\centering
 	\caption{ Optimal Bayesian life-testing plans $(n,r,l,T_1,T_2)$ under Type-II UHCS with Prior 2}
 		
 		% Optimum Bayesian design under Type-II UHCS for the hyper parameters $a_1=5; b_1=10; p_2=0.01; q_2=50$.}
 	\begin{center}
 		\begin{tabular}{cccc}\toprule		
 			
 			$(C_f, C_t)$& $C_b$ & $(n, r, l, T_1, T_2)$  & $D$-optimal \\
 			\toprule	
 			(10,15)&250   & (18, 14, 7, 0.8887, 1.7775)    &  5.1840\\ 
 			&300   &  (22, 16, 8, 0.9320, 1.8641)   & 5.5697 \\ 
 			& 350  &  (25, 19, 10, 1.0558, 2.1116)   &  5.8435\\ 
 			& 400  &  (28, 23, 12, 1.1748, 2.3497)   & 6.1595 \\ 
 			& 450  &   (32, 25, 13, 1.1594, 2.3189)  &  6.4834\\ 
 			& 500  & (35, 28, 14, 1.2939, 2.5878)    &  6.7128\\
 			
 			\bottomrule
 		\end{tabular}
 	\end{center}
 \end{table}

\section{Conclusion}
\paragraph{}
This article establishes the explicit expressions of the expected number of failures, expected duration of testing and the Fisher information matrix for the unknown parameters of the underlying lifetime model when the Type-II unified hybrid censoring scheme is employed. Using these quantities, optimal Bayesian design of life-testing plans are also discussed in this article. One constrained optimality criterion is demonstrated through numerical example. Consideration of other relevant optimality criteria under various lifetime models could be possible further extensions of the present article as future research interests. Present research is going on in this direction and we hope to report our findings in future articles.

.

\bibliographystyle{unsrt}
\bibliography{BaysWarranty}

\end{document}